\documentclass[a4paper,12pt]{article}

\usepackage{setspace}
\usepackage[a4paper,left=4cm, right=4cm, top=5.2cm, bottom=5.3cm]{geometry}
\usepackage{amsmath,amssymb}
\usepackage{amsthm}
\usepackage{array,graphicx}
\usepackage{booktabs}
\usepackage{pifont}
\usepackage{tabu}
\usepackage{longtable}
\usepackage{xcolor}
\usepackage{tcolorbox}
\usepackage{textcomp}
\usepackage{caption}
\usepackage{tikz}
\usepackage[normalem]{ulem}

\usepackage[utf8]{inputenc}
\usepackage[english]{babel}

\newtheorem{theorem}{Theorem}
\newtheorem{lemma}{Lemma}

\theoremstyle{definition}
\newtheorem{definition}{Definition}

\title{\textbf{\LARGE Transitivity of implicative aBE algebras}} 

\author{ Denis Zelent}

\date{}

\begin{document}
\maketitle

\noindent
\textbf{Abstract.} {\footnotesize 
We prove that every implicative aBE algebra satisfies the transitivity property. This means that every implicative aBE algebra is a Tarski algebra, and thus is also a commutative BCK algebra. \\
}

\bigskip\bigskip\bigskip

BCK-algebras were introduced by Y. Imai and K. Iséki in \cite{Imai1966}, and were later generalized by H. S. Kim and Y. H. Kim in the definition of BE algebras \cite{Kim07}. In 2016, A. Iorgulescu defined numerous new generalizations, such as aBE, BE** and aBE** algebras \cite{Ior16}. The definition of implicativity for a BCK algebra was introduced by K. Iséki and S. Tanaka in \cite{Iseki78}, and was studied for various generalizations of BCK algebras by A. Walendziak in \cite{Wal18, Wal}. In \cite{Wal} author posed an open problem: whether there exists an implicative aBE algebra that does not satisfy the transitivity property (Open Problem 3.21). In this paper, we prove that this property is satisfied by every implicative aBE algebra, which gives a negative answer to the posed question. This result gives us immediately a few implications about implicative aBE algebras. Firstly, every implicative aBE algebra is a Tarski algebra \cite[Corollary 3.19]{Wal}. This means that every implicative aBE algebra is an implicative aBE** algebra \cite[Theorem 3.18]{Wal}, and thus it is also a commutative BCK algebra \cite[Corollary 3.13]{Wal}.

\bigskip

We start by introducing the definitions of implicative aBE algebra and transitivity property.

\begin{definition} \label{bckdef}
An \textsl{implicative aBE algebra} is an algebra of the form  (X, $\rightarrow$, $1$), where X is the non-empty set with a designated element $1$ and the arrow as a binary operation satisfying the following axioms:
\begin{equation}\label{bckdef1}
    1\rightarrow x = x
\end{equation}
\begin{equation}\label{bckdef2}
    x\rightarrow 1 = 1
\end{equation}
\begin{equation}\label{bckdef3}
    x\rightarrow x = 1
\end{equation}
\begin{equation}\label{bckdef4}
    x\rightarrow(y \rightarrow z) = y\rightarrow (x\rightarrow z)
\end{equation}
\begin{equation}\label{bckdef5}
    x\rightarrow y = y\rightarrow x = 1 \implies x=y
\end{equation}
\begin{equation}\label{bckdef6}
    (x\rightarrow y)\rightarrow x = x
\end{equation}
\end{definition}

\begin{definition}
    We say that an (implicative aBE) algebra (X, $\rightarrow$, $1$) satisfies the \textsl{transitivity property} if and only if it satisfies
    \begin{equation}
        x\rightarrow y = y\rightarrow z = 1 \implies x\rightarrow z = 1.
    \end{equation}
\end{definition}

To prove that every implicative aBE algebra satisfies the transitivity property, we would first state and prove some relevant lemmas. 
% Every one of them is about implicative aBE algebras.
\begin{lemma} If (X, $\rightarrow$, $1$) is an implicative aBE algebra, then it satisfies:
    \begin{equation}\label{eq136}
        \text{a)   } x = y\rightarrow x \text{ or } (y\rightarrow x)\rightarrow x \neq 1
    \end{equation}
    \begin{equation}\label{eq1.538}
        \text{b)   } x = (x\rightarrow y)\rightarrow y \text{ or } ((x\rightarrow y)\rightarrow y)\rightarrow x \neq 1.
    \end{equation}
\end{lemma}
\begin{proof}
    Firstly, we note that (\ref{bckdef5}) implies 
    $$x = y \text{ or } x\rightarrow y \neq 1 \text{ or } y\rightarrow x \neq 1.$$
    To prove (\ref{eq136}) we use that fact for $x$ and $y\rightarrow x$: 
    $$x = y\rightarrow x \text{ or } x\rightarrow (y\rightarrow x) \neq 1 \text{ or } (y\rightarrow x)\rightarrow x \neq 1.$$
    Noticing that $x\rightarrow (y\rightarrow x) \stackrel{(4)}{=} y\rightarrow (x\rightarrow x) \stackrel{(3)}{=} y\rightarrow 1 \stackrel{(2)}{=} 1$ concludes the proof of the first equation.
    Proof of \ref{eq1.538} is analogical but with $x$ and $(x\rightarrow y)\rightarrow y$ and observation that $x\rightarrow((x\rightarrow y)\rightarrow y) \stackrel{(4)}{=} (x\rightarrow y)\rightarrow(x\rightarrow y) \stackrel{(3)}{=} 1$.
\end{proof}

\begin{lemma} If (X, $\rightarrow$, $1$) is an implicative aBE algebra, then it satisfies:
    \begin{equation}\label{eq199}
        x\rightarrow y = (z\rightarrow x)\rightarrow (x\rightarrow y)
    \end{equation}
\end{lemma}
\begin{proof}
    $x\rightarrow y \stackrel{(6)}{=} ((x\rightarrow y)\rightarrow (z\rightarrow x))\rightarrow (x\rightarrow y) \stackrel{(4)}{=} (z\rightarrow ((x\rightarrow y)\rightarrow x))\rightarrow (x\rightarrow y) \stackrel{(6)}{=} (z\rightarrow x)\rightarrow (x\rightarrow y)$.
\end{proof}

\begin{lemma} If (X, $\rightarrow$, $1$) is an implicative aBE algebra, then it satisfies:
    \begin{equation}\label{eq1298}
        ((y\rightarrow x)\rightarrow z)\rightarrow t = ((y\rightarrow x)\rightarrow z)\rightarrow ((x\rightarrow z)\rightarrow t)
    \end{equation}
\end{lemma}
\begin{proof}
    Set $a:=((y\rightarrow x)\rightarrow z)$.
    We have $x\rightarrow a = x\rightarrow ((y\rightarrow x)\rightarrow z) \stackrel{(4)}{=} (y\rightarrow x)\rightarrow (x\rightarrow z) \stackrel{(\ref{eq199})}{=} x\rightarrow z$. Then $a\rightarrow t \stackrel{(\ref{eq199})}{=} (x\rightarrow a)\rightarrow (a\rightarrow t)\stackrel{(4)}{=} a\rightarrow((x\rightarrow a)\rightarrow t) = a\rightarrow ((x\rightarrow z)\rightarrow t).$

    % $((y\rightarrow x)\rightarrow z)\rightarrow t \stackrel{(\ref{eq199})}{=} (x\rightarrow ((y\rightarrow x)\rightarrow z))\rightarrow (((y\rightarrow x)\rightarrow z)\rightarrow t) \stackrel{(4)}{=} ((y\rightarrow x)\rightarrow z)\rightarrow ((x\rightarrow ((y\rightarrow x)\rightarrow z))\rightarrow t) \stackrel{(4)}{=} ((y\rightarrow x)\rightarrow z)\rightarrow (((y\rightarrow x)\rightarrow (x\rightarrow z))\rightarrow t)\stackrel{(\ref{eq199})}{=} ((y\rightarrow x)\rightarrow z)\rightarrow ((x\rightarrow z)\rightarrow t).$
\end{proof}

\begin{lemma} If (X, $\rightarrow$, $1$) is an implicative aBE algebra, then it satisfies:
    \begin{equation}\label{eq368}
        (((x\rightarrow y)\rightarrow z)\rightarrow y)\rightarrow (x\rightarrow y) = 1
    \end{equation}
\end{lemma}
\begin{proof}
    Set $a := (x\rightarrow y)\rightarrow z$. We get $(a\rightarrow y)\rightarrow (x\rightarrow y) \stackrel{(6)}{=} (a\rightarrow y) \rightarrow (a \rightarrow (x \rightarrow y)) \stackrel{(4)}{=} (a \rightarrow y) \rightarrow (x \rightarrow (a \rightarrow y)) \stackrel{(4)}{=} x \rightarrow ((a \rightarrow y) \rightarrow (a \rightarrow y)) \stackrel{(3)}{=} x \rightarrow 1 \stackrel{(2)}{=} 1$.
    
    % $(((x\rightarrow y)\rightarrow z)\rightarrow y)\rightarrow (x\rightarrow y) \stackrel{(6)}{=} (((x\rightarrow y)\rightarrow z)\rightarrow y)\rightarrow (((x\rightarrow y)\rightarrow z)\rightarrow (x\rightarrow y)) \stackrel{(4)}{=} (((x\rightarrow y)\rightarrow z)\rightarrow y)\rightarrow (x\rightarrow (((x\rightarrow y)\rightarrow z)\rightarrow y)) \stackrel{(4)}{=} x\rightarrow ((((x\rightarrow y)\rightarrow z)\rightarrow y)\rightarrow (((x\rightarrow y)\rightarrow z)\rightarrow y)) \stackrel{(3)}{=} x\rightarrow 1 \stackrel{(2)}{=} 1$.
\end{proof}

\begin{lemma} If (X, $\rightarrow$, $1$) is an implicative aBE algebra, then it satisfies:
    \begin{equation}\label{eq3.775}
        (((((x\rightarrow y)\rightarrow y)\rightarrow x)\rightarrow y)\rightarrow y) = 1
    \end{equation}
\end{lemma}
\begin{proof}
    $(((((x\rightarrow y)\rightarrow y)\rightarrow x)\rightarrow y)\rightarrow y) \stackrel{(\ref{eq1298})}{=} ((((x\rightarrow y)\rightarrow y)\rightarrow x)\rightarrow y)\rightarrow ((x\rightarrow y)\rightarrow y) \stackrel{(\ref{eq368})}{=} 1$.
\end{proof}

\begin{lemma} If (X, $\rightarrow$, $1$) is an implicative aBE algebra, then it satisfies:
    \begin{equation}\label{eq3.999}
        y = (((x\rightarrow y)\rightarrow y)\rightarrow x)\rightarrow y
    \end{equation}
\end{lemma}
\begin{proof}
    Follows directly from (\ref{eq136}) and (\ref{eq3.775}).
\end{proof}

\begin{lemma} If (X, $\rightarrow$, $1$) is an implicative aBE algebra, then it satisfies:
    \begin{equation}\label{eq44}
        x\rightarrow y = x\rightarrow(((x\rightarrow y)\rightarrow z)\rightarrow y)
    \end{equation}
\end{lemma}
\begin{proof}
    $x\rightarrow y \stackrel{(6)}{=} ((x\rightarrow y)\rightarrow z)\rightarrow (x\rightarrow y) \stackrel{(4)}{=} x\rightarrow (((x\rightarrow y)\rightarrow z)\rightarrow y) $.
\end{proof}

\begin{lemma} If (X, $\rightarrow$, $1$) is an implicative aBE algebra, then it satisfies:
    \begin{equation}\label{eq3.112}
        ((x\rightarrow y)\rightarrow y)\rightarrow x = ((x\rightarrow y)\rightarrow y)\rightarrow (y\rightarrow x)
    \end{equation}
\end{lemma}
\begin{proof}
    $((x\rightarrow y)\rightarrow y)\rightarrow x \stackrel{(\ref{eq44})}{=} ((x\rightarrow y)\rightarrow y)\rightarrow ((((x\rightarrow y)\rightarrow y)\rightarrow x)\rightarrow y)\rightarrow x) \stackrel{(\ref{eq3.999})}{=} ((x\rightarrow y)\rightarrow y)\rightarrow (y\rightarrow x)$.
\end{proof}

\begin{lemma} If (X, $\rightarrow$, $1$) is an implicative aBE algebra, then it satisfies:
    \begin{equation}\label{eq3.889}
        y\rightarrow x = ((x\rightarrow y)\rightarrow y)\rightarrow x
    \end{equation}
\end{lemma}
\begin{proof}
    Follows directly from (\ref{eq3.112}) and (\ref{eq199}).
\end{proof}

\begin{lemma}\label{th1} If (X, $\rightarrow$, $1$) is an implicative aBE algebra, then it satisfies:
    \begin{equation}
        (x\rightarrow y)\rightarrow y = x \text{ or } y\rightarrow x \neq 1.
    \end{equation}
\end{lemma}
\begin{proof}
    This theorem is now just a combination of (\ref{eq1.538}) together with (\ref{eq3.889}).
\end{proof}

\begin{theorem} (Transitivity of implicative aBE algebras)
    Every implicative aBE algebra satisfies the transitivity property.
\end{theorem}
\begin{proof}
    Assume by contradiction that there exists algebra for which this condition does not hold, i.e. exists $a,b,c$ such that $a\rightarrow b = 1, b \rightarrow c = 1$, but $a\rightarrow c \neq 1$. Then from Lemma \ref{th1}, since $b\rightarrow c = 1$, we must have $c = (c\rightarrow b)\rightarrow b$. But multiplying both sides from left by $a$ gives us 
    $$a\rightarrow c = a\rightarrow ((c\rightarrow b)\rightarrow b) \stackrel{(4)}{=} (c\rightarrow b)\rightarrow (a\rightarrow b) = (c\rightarrow b)\rightarrow 1 \stackrel{(2)}{=} 1 $$
    which is a contradiction with $a\rightarrow c \neq 1$.
\end{proof}

Denis Zelent\\
Department of Mathematical Sciences\\
Norwegian University of Science and Technology\\
7491 Trondheim, Norway\\
e-mail: zelden99@zohomail.eu
\end{document}